\documentclass[12pt]{amsart}
\usepackage{amssymb}
\usepackage{amsmath}
\usepackage{amscd}
\usepackage{verbatim}
\usepackage[normalem]{ulem}
\usepackage{amssymb,amsfonts,pstricks,epsf}
\usepackage{graphicx}
\usepackage{subcaption}
\usepackage{epsfig}
\usepackage{color}
\usepackage{stmaryrd}
\usepackage{mathrsfs}
\usepackage{marginnote}
\usepackage{geometry}
\newtheorem{theorem}{Theorem}[section]

\newtheorem{proposition}[theorem]{Proposition}
\newtheorem{corollary}[theorem]{Corollary}

\newtheorem{definition}[theorem]{Definition}
\newtheorem{problem}[theorem]{Problem}
\newtheorem{example}{Example}[section]

\definecolor{plum}{rgb}{1.0, 0.0, 1.0}

\DeclareMathOperator{\osc}{\textup{osc}}
\DeclareMathOperator{\esssup}{\textup{ess\ sup}}
\DeclareMathOperator{\essinf}{\textup{ess\ inf}}

\makeatletter
\@namedef{subjclassname@2020}{%
  \textup{2020} Mathematics Subject Classification}
\makeatother

\begin{document}

\title{Extremizing Temperature Functions of Rods with Robin Boundary Conditions}

\author{Jeffrey J. Langford and Patrick McDonald}

\address{Department of Mathematics, Bucknell University, Lewisburg, Pennsylvania 17837}

\email{jeffrey.langford@bucknell.edu}

\address{Division of Natural Science, New College of Florida, Sarasota, FL 34243}

\email{mcdonald@ncf.edu}

\date{\today}

\begin{abstract}
We compare the solutions of two one-dimensional Poisson problems on an interval with Robin boundary conditions, one with given data, and one where the data has been symmetrized. When the Robin parameter is positive and the symmetrization is symmetric decreasing rearrangement, we prove that the solution to the symmetrized problem has larger increasing convex means. When the Robin parameter equals zero (so that we have Neumann boundary conditions) and the symmetrization is decreasing rearrangement, we similarly show that the solution to the symmetrized problem has larger convex means.
\end{abstract}

\keywords{Symmetrization, comparison theorems, Poisson's equation, Robin boundary conditions}
 
\subjclass[2020]{Primary 34B08; Secondary 34C10}

\maketitle

\section{Introduction: Physical Motivation and Main Results}
Our paper is motivated by the following physical problem:

\begin{problem}\label{Prob:1dRod}
Consider a metal rod of length $\ell$. To half the locations on the rod, heat is generated uniformly, while on the remaining half of the rod, heat is neither generated nor absorbed. If the rod's ends are frozen at zero temperature, where should we place the heat sources to maximize the hottest steady-state temperature across the rod?
\end{problem}
Several possible arrangements appear in Figure 1 below. Heat is generated in the white regions while heat is neither generated nor absorbed in the gray regions.

\begin{figure}[h]
\centering
\begin{subfigure}[t]{.4\textwidth}
\centering
\includegraphics[width=\linewidth]{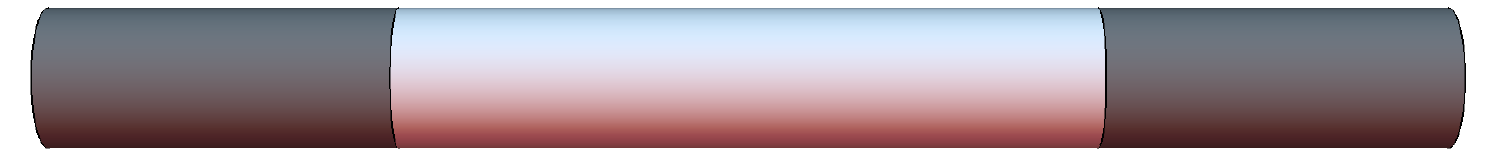}
        \caption{}\label{fig:fig_a}
\end{subfigure}
\begin{subfigure}[t]{.4\textwidth}
\centering
\includegraphics[width=\linewidth]{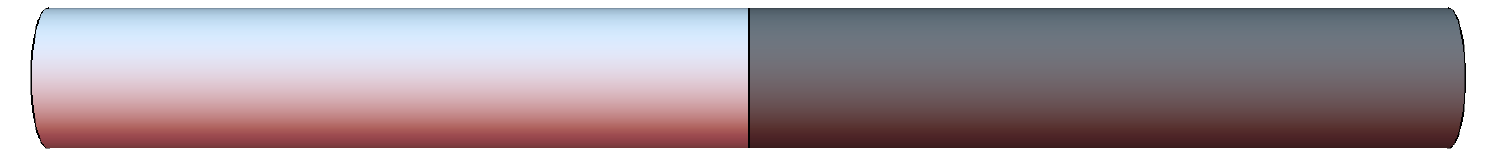}
\caption{}\label{fig:fig_b}
\end{subfigure}

\medskip
\begin{subfigure}[t]{.4\textwidth}
\centering
\vspace{0pt}
\includegraphics[width=\linewidth]{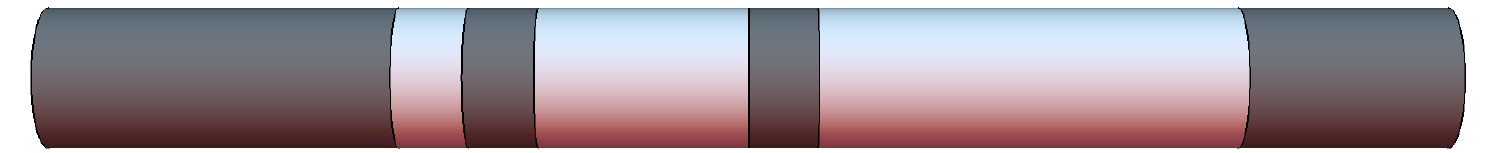}
\caption{}\label{fig:fig_c}
\end{subfigure}
\begin{subfigure}[t]{.4\textwidth}
\centering
\vspace{0pt}
\includegraphics[width=\linewidth]{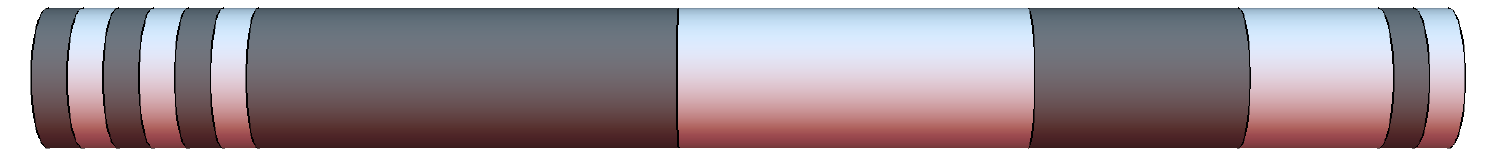}
\caption{}\label{fig:fig_c}
\end{subfigure}
\begin{minipage}[t]{1\textwidth}
\caption{Four possible heat source arrangements for Problem \ref{Prob:1dRod}.}
\end{minipage}

\end{figure}

The solution to Problem \ref{Prob:1dRod} follows from a celebrated result in symmetrization known as Talenti's Theorem \cite{Talenti}. To understand Talenti's solution\footnote{Although Talenti's work in \cite{Talenti} explicitly assumes that the dimension $n\geq2$, the result still holds in dimension $1$. For a different approach to comparison theorems that yields the same result in all dimensions, see Corollary 3 of \cite{Baernstein Cortona Volume} or Theorem 10.10 of \cite{Barenstein Star Function in Complex Analysis} and Corollary \ref{Cor:DComp} below.}, we write out the mathematical formulation of Problem \ref{Prob:1dRod}. Suppose the rod is located along the interval $\left[-\frac{\ell}{2},\frac{\ell}{2}\right]$ and let $E\subseteq [-\frac{\ell}{2},\frac{\ell}{2}]$ denote the locations of the heat sources. Then the steady-state temperature function $u$ satisfies the one-dimensional Poisson problem
\begin{equation}\label{eq:1duTalenti}
-u''=\chi_E  \quad \textup{in} \quad \left(-\frac{\ell}{2},\frac{\ell}{2}\right), \qquad u\left(-\frac{\ell}{2}\right)=u\left(\frac{\ell}{2}\right)=0,
\end{equation}
where $\chi_E$ denotes the characteristic function of the set $E$. Talenti's Theorem compares the solution $u$ in \eqref{eq:1duTalenti} to the solution $v$ of a problem that has been ``symmetrized.'' Specifically, let $v$ solve the Poisson problem
\begin{equation*}
-v''=\chi_{\left(-\frac{|E|}{2},\frac{|E|}{2}\right)}  \quad \textup{in} \quad \left(-\frac{\ell}{2},\frac{\ell}{2}\right), \qquad v\left(-\frac{\ell}{2}\right)=v\left(\frac{\ell}{2}\right)=0,
\end{equation*}
where $|E|$ denotes the length of $E$; in this case $\chi_{\left(-\frac{|E|}{2},\frac{|E|}{2}\right)}$ is called the \emph{symmetric decreasing rearrangement} of $\chi_E$ (for a precise definition, see Definition \ref{def:decrearr}). Talenti showed that the temperature functions $u$ and $v$ compare through their convex means, that is,
\begin{equation}\label{eq:convuv}
\int_{-\frac{\ell}{2}}^{\frac{\ell}{2}}\phi(u)\,dx\leq \int_{-\frac{\ell}{2}}^{\frac{\ell}{2}}\phi(v)\,dx
\end{equation}
for each convex increasing function $\phi:\mathbb{R} \to \mathbb{R}$. Since $u$ and $v$ are concave functions, they are minimized at the ends of the interval $\left[-\frac{\ell}{2},\frac{\ell}{2}\right]$, where both functions vanish. That is,
\begin{equation}\label{eq:infuv}
\min_{\left[-\frac{\ell}{2},\frac{\ell}{2}\right]} u=\min_{\left[-\frac{\ell}{2},\frac{\ell}{2}\right]} v=0.
\end{equation}
Thus $u$ and $v$ are nonnegative. Taking $\phi(x)=\chi_{[0,\infty)}(x)\cdot x^p$ in \eqref{eq:convuv} gives
\[
\|u\|_{L^p\left[-\frac{\ell}{2},\frac{\ell}{2}\right]}\leq \|v\|_{L^p\left[-\frac{\ell}{2},\frac{\ell}{2}\right]}, \qquad 1\leq p<+\infty,
\]
and sending $p\to +\infty$ shows
\begin{equation}\label{eq:supuv}
\max_{\left[-\frac{\ell}{2},\frac{\ell}{2}\right]} u \leq \max_{\left[-\frac{\ell}{2},\frac{\ell}{2}\right]} v.
\end{equation}
Talenti's Theorem thus says that the hottest temperature in Problem \ref{Prob:1dRod} is maximized when the heat sources are centrally gathered in the middle of the rod as in (A) of Figure 1. Note also that equations \eqref{eq:infuv} and \eqref{eq:supuv} show that 
\[
\underset{\left[-\frac{\ell}{2},\frac{\ell}{2}\right]}{\textup{osc}}\ u \leq \underset{\left[-\frac{\ell}{2},\frac{\ell}{2}\right]}{\textup{osc}}\ v,
\]
where $\textup{osc}=\max - \min$ denotes the oscillation, or temperature gap (the difference between the rod's largest and smallest temperatures). Thus, arrangement (A) not only maximizes the rod's hottest temperature but also its temperature gap.

The present paper is motivated by a simple question: What happens if we consider analogues of Problem \ref{Prob:1dRod} with other boundary conditions? To start, we might consider a situation where the ends of the bar interact with the outside environment. For example, imagine that each end of the bar is submerged in a large bath of fluid with temperature zero.  Newton's law of cooling then says that the heat flux is proportional to the temperature at each end of the rod. This physical setting yields boundary conditions known as Robin boundary conditions. Thus, we ask:

\begin{problem}\label{Prob:1dRodRobin}
With the same setup as Problem \ref{Prob:1dRod} for a rod with Robin boundary conditions, where should we locate the heat sources to maximize the hottest steady-state  temperature?
\end{problem}

We solve Problem \ref{Prob:1dRodRobin} by proving a one-dimensional comparison principle for Robin problems in the spirit of Talenti. The result stated below is normalized so the length of the bar equals $2\pi$, but the the result holds for any interval. Specifically, we prove:

\begin{theorem}[ODE Robin Comparison Principle]\label{Th:RComp}
Let $0\leq f\in L^1[-\pi,\pi]$ and $\alpha>0$. Suppose $u$ and $v$ solve the Poisson problems
\begin{align*}
-u''&=f  \quad \textup{in} \quad (-\pi,\pi), \qquad -u'(-\pi)+\alpha u(-\pi)=u'(\pi)+\alpha u(\pi)=0,\\
-v''&=f^{\#}  \quad \textup{in} \quad (-\pi,\pi), \qquad -v'(-\pi)+\alpha v(-\pi)=v'(\pi)+\alpha v(\pi)=0,
\end{align*}
with $f^{\#}$ the symmetric decreasing rearrangement of $f$. Then
\[
\int_{-\pi}^{\pi}\phi(u)\,dx\leq \int_{-\pi}^{\pi}\phi(v)\,dx
\]
for each increasing convex function $\phi:\mathbb{R}\to \mathbb{R}$. In particular,
\begin{equation}\label{eq:LpRobin}
\|u\|_{L^p[-\pi,\pi]} \leq \|v\|_{L^p[-\pi,\pi]},\qquad 1\leq p\leq +\infty.\\
\end{equation}
\end{theorem}

To resolve Problem \ref{Prob:1dRodRobin}, let $E\subseteq [-\pi,\pi]$ denote the locations of the heat sources and let $u$ and $v$ denote the corresponding temperature functions for the Robin problems of Theorem \ref{Th:RComp}:
\begin{align*}
-u''&=\chi_E \quad \textup{in} \quad (-\pi,\pi), \qquad -u'(-\pi)+\alpha u(-\pi)=u'(\pi)+\alpha u(\pi)=0,\\
-v''&= \chi_{\left(-\frac{\pi}{2},\frac{\pi}{2}\right)} \quad \textup{in} \quad (-\pi,\pi), \qquad -v'(-\pi)+\alpha v(-\pi)=v'(\pi)+\alpha v(\pi)=0.
\end{align*}
The proof of Theorem \ref{Th:RComp} shows that $u$ and $v$ are nonnegative, thus taking $p=+\infty$ in \eqref{eq:LpRobin} shows
\[
\max_{\left[-\pi,\pi\right]} u \leq \max_{\left[-\pi,\pi\right]} v.
\]
Thus, as in Problem \ref{Prob:1dRod}, Problem \ref{Prob:1dRodRobin} is resolved with an arrangement of heat sources analogous to (A) in Figure 1.  In fact, in Corollary \ref{Cor:DComp}, we prove that the corresponding Dirichlet result follows from Theorem \ref{Th:RComp}.  Unlike the Dirichlet setting, however, the temperature gap does not necessarily increase under symmetric decreasing rearrangement; see Example \ref{ex:osc} and Proposition \ref{prop:robingapint}.

We also address the situation where the ends of the rod are perfectly insulated. In this setting, we cannot consider a verbatim analogue of Problem \ref{Prob:1dRod}, since perfect insulation requires the presence of both heat sinks and sources. The temperature function, moreover, is unique only up to an additive constant. Thus, we ask:

\begin{problem}\label{Prob:1dRodNeumann}
Suppose half of a given rod is heated uniformly, while on the complimentary half, heat is absorbed uniformly. If the rod's ends are perfectly insulated, where should we place the heat sources and sinks to maximize the hottest steady-state temperature across the rod, assuming the temperature has zero mean?
\end{problem}

Again, we solve Problem \ref{Prob:1dRodNeumann} by proving a Talenti-style comparison principle. We normalize and assume the rod has length $\pi$, but as before, the result holds for any interval. We prove:

\begin{theorem}[ODE Neumann Comparison Principle]\label{Th:NComp}
Let $f\in L^1[0,\pi]$ have zero mean and suppose $u$ and $v$ solve the Poisson problems
\begin{align*}
-u''&=f  \quad \textup{in} \quad (0,\pi), \qquad u'(0)=u'(\pi)=0,\\
-v''&=f^{\ast}  \quad \textup{in} \quad (0,\pi), \qquad v'(0)=v'(\pi)=0,
\end{align*}
with $f^{\ast}$ the decreasing rearrangement of $f$. If $u$ and $v$ both have zero mean, then
\[
\int_{0}^{\pi}\phi(u)\,dx\leq \int_{0}^{\pi}\phi(v)\,dx
\]
for each convex function $\phi:\mathbb{R}\to \mathbb{R}$. In particular,
\begin{align*}
\|u\|_{L^p[0,\pi]} &\leq \|v\|_{L^p[0,\pi]},\qquad 1\leq p\leq +\infty,\\
\max_{[0,\pi]}u \leq \max_{[0,\pi]}v,\qquad \min_{[0,\pi]}v  &\leq \min_{[0,\pi]}u, \qquad \underset{[0,\pi]}{\textup{osc}} \ u \leq \underset{[0,\pi]}{\textup{osc}} \ v.
\end{align*}
\end{theorem}

To resolve Problem \ref{Prob:1dRodNeumann}, let $E \subseteq [0,\pi]$ denote the locations of the heat sources and let $u$ and $v$ denote the temperature functions for the Neumann problems of Theorem \ref{Th:NComp}:
\begin{align*}
-u''&=\chi_E-\chi_{[0,\pi]\setminus E}  \quad \textup{in} \quad (0,\pi), \qquad u'(0)=u'(\pi)=0,\\
-v''&=\chi_{\left[0,\frac{\pi}{2}\right)}-\chi_{\left[\frac{\pi}{2},\pi\right]}  \quad \textup{in} \quad (0,\pi), \qquad v'(0)=v'(\pi)=0.
\end{align*}
According to Theorem \ref{Th:NComp}, the maximum temperature and temperature gap increase, and the minimal temperature decreases under decreasing rearrangement:
\[
\max_{[0,\pi]}u \leq \max_{[0,\pi]}v,\qquad \min_{[0,\pi]}v  \leq \min_{[0,\pi]}u, \qquad \underset{[0,\pi]}{\textup{osc}} \ u \leq \underset{[0,\pi]}{\textup{osc}} \ v.
\]
Thus Problem \ref{Prob:1dRodNeumann} is resolved by choosing an arrangement of sources and sinks analogous to (B) in Figure 1, only here, white areas represent heat sources and gray areas represent heat sinks.

Taken in sum, Theorems \ref{Th:RComp} and \ref{Th:NComp} reveal a striking difference in the behavior of source functions that induce large temperature functions (interpreted in the sense of convex means). With the Neumann problem, one takes full advantage of the insulated ends, sweeping the greatest sources to one end of the bar and greatest sinks to the opposite end. However, the instant any heat energy is allowed to escape through the bar's ends and the Robin regime is entered, the arrangement switches and we instead move the greatest sources towards the middle of the bar and push the weakest sources out towards the ends.

The results of our paper are examples of comparison theorems for differential equations. To place our work in the existing literature, we recall that the first major comparison result, as mentioned above, is due to Talenti \cite{Talenti}, who compared the solutions of two Poisson problems with Dirichlet boundary conditions and nonnegative source, $f$, namely
\[
\begin{array}{rclccccrclcc}
-\Delta u & = & f & \text{in} & \Omega, &  &  & -\Delta v & = & f^{\#} & \text{in} & \Omega^{\#},\\
u & = & 0 & \text{on} & \partial \Omega, &  &  &v & = & 0 & \text{on} & \partial \Omega^{\#}.
\end{array}
\]
Here, $\Omega \subseteq \mathbb{R}^n$ is a bounded Lipschitz domain with $n\geq 2$, $0\leq f\in L^2(\Omega)$, $\Omega^{\#} \subseteq \mathbb{R}^n$ is the open ball centered at $0$ with the same volume as $\Omega$, and $f^{\#}$ denotes the symmetric decreasing rearrangement of $f$, a radially decreasing function on $\Omega^{\#}$ whose upper level sets have the same volume as those of $f$, meaning $|\{x\in \Omega:f(x)>t\}|=|\{x\in \Omega^{\#}:f^{\#}(x)>t\}|$ for $t\in \mathbb{R}$. Talenti showed that the solutions $u$ and $v$ compare via their symmetric decreasing rearrangements through the inequality
\[
u^{\#} \leq v \quad \textup{in }\Omega^{\#}.
\]

 
The history of comparison phenomena that followed Talenti's original work is long and the results are the subject of many articles. Fortunately, Talenti has prepared a thorough survey of the material (up to 2016). We direct the reader interested in this important background to  \cite{TalentiSurvey} and the references therein. 

Since the publication of Talenti's survey, authors have begun to turn their attention to comparison principles for Robin problems.  As an example relevant to our work, in \cite{ANT} Alvino, Nitsch, and Trombetti consider the exact same setup addressed by Talenti \cite{Talenti} and mentioned above, but impose Robin boundary conditions rather than Dirichlet boundary conditions. That is, for $\alpha>0$, $\Omega \subseteq \mathbb{R}^n$ a bounded Lipschitz domain with $n\geq 2$, and $0\leq f\in L^2(\Omega)$, they consider the problems
\[
\begin{array}{rclccccrclcc}
-\Delta u & = & f & \text{in} & \Omega, &  &  & -\Delta v & = & f^{\#} & \text{in} & \Omega^{\#},\\
\frac{\partial u}{\partial \nu}+\alpha u & = & 0 & \text{on} & \partial \Omega, &  &  &\frac{\partial v}{\partial \nu}+\alpha v & = & 0 & \text{on} & \partial \Omega^{\#},
\end{array}
\]
with $\frac{\partial}{\partial \nu}$ the outer normal derivative and $\#$ the symmetric decreasing rearrangement. The authors show that Talenti's conclusion $u^{\#}\leq v$ in $\Omega^{\#}$ fails in general, but that $u$ and $v$ compare via their Lorentz norms. In dimension $n=2$, they show that the $L^1$- and $L^2$-norms of $u$ are dominated by those of $v$, and when $f=1$, that $u^{\#} \leq v$ in $\Omega^{\#}$. These results are extended in the subsequent work of Alvino, Chiacchio, Nitsch, and Trombetti \cite{ACNT}. In related work \cite{AGM}, Amato, Gentile, and Masiello, generalize results of \cite{ANT} to a nonlinear setting, replacing the Laplacian with the $p$-Laplace operator.

In addition to the the results of \cite{ACNT}, \cite{ANT}, and \cite{AGM}, in \cite{Langford4} the first author studies Poisson problems of the form
\[
\begin{array}{rclccccrclcc}
-\Delta u & = & f & \text{in} & A, &  &  & -\Delta v & = & f^{\#} & \text{in} & A,\\
\frac{\partial u}{\partial \nu}+\alpha u & = & 0 & \text{on} & \partial A, &  &  &\frac{\partial v}{\partial \nu}+\alpha v & = & 0 & \text{on} & \partial A,
\end{array}
\]
where $A\subseteq \mathbb{R}^n$ is a spherical shell (the region between two concentric spheres), $\alpha>0$, $f\in L^2(A)$ and $\#$ is the {\it cap symmetrization}. (To cap symmetrize a function $f:A\to \mathbb{R}$, one applies the spherical rearrangement (the analogue of the symmetric decreasing rearrangement on the sphere) to each of $f$'s radial slice functions). The author shows that the solutions $u$ and $v$ compare through their convex means:
\[
\int_{A}\phi(u)\,dx\leq \int_{A}\phi(v)\,dx
\]
for each convex function $\phi:\mathbb{R} \to \mathbb{R}$. The author obtains similar results for $\alpha=0$ (the Neumann problem), assuming $f$, $u$, and $v$ all have zero mean. (For related work on the Neumann problem, see \cite{Langford1}, \cite{Langford2}, and \cite{Langford3}).

To the best of our knowledge, references \cite{ACNT}, \cite{ANT}, \cite{AGM}, and \cite{Langford4} comprise all that has appeared in print to addresses Robin comparison principles for differential equations in the spirit of Talenti. Thus, our work adds an interesting contribution to this new direction in the study of comparison principles.

The rest of this note is organized as follows. In Section 2 we discuss existence and uniqueness results for the Poisson problems of Theorems \ref{Th:RComp} and \ref{Th:NComp}, so that our paper may be self-contained. We then discuss Robin Green's functions and relevant rearrangement inequalities needed to prove Theorems \ref{Th:RComp} and \ref{Th:NComp}. In Section 3, we prove our paper's main results.

\section{Background}

Since the goal of our paper is to compare the solutions of one-dimensional Poisson problems with Robin and Neumann boundary conditions, we begin with two existence and uniqueness results. These results are stated on the interval $[-\pi,\pi]$ for convenience, but they hold for any interval.

\begin{proposition}[Robin Existence and Uniqueness]\label{Prop:Runiq}
Let $f\in L^1[-\pi,\pi]$ and $\alpha>0$. A unique $u\in C^1[-\pi,\pi]$ exists satisfying
\begin{itemize}
\item[1.] $u'$ is absolutely continuous on $[-\pi,\pi]$.
\item[2.] $-u''=f$ a.e. on $(-\pi,\pi)$.
\item[3.] $-u'(-\pi)+\alpha u(-\pi)=u'(\pi)+\alpha u(\pi)=0$.
\end{itemize}
\end{proposition}

\begin{proof}
We first establish uniqueness. Suppose $u$ and $v$ both satisfy all the properties listed above, and let $w=u-v$. Since $w'$ is absolutely continuous, for each $x\in[-\pi,\pi]$ we have
\[
w'(x)=w'(x)-w'(-\pi)+\alpha w(-\pi)=\int_{-\pi}^x(-f+f)\,dy+\alpha w(-\pi)=\alpha w(-\pi).
\]
Thus, $w(x)=\alpha w(-\pi)x+b$ for some constant $b$. The equations $-w'(-\pi)+\alpha w(-\pi)=w'(\pi)+\alpha w(\pi)=0$ imply that $\alpha w(-\pi)=b=0$, and so $u\equiv v$. For existence, we simply take
\[
u(x)=-\int_{-\pi}^x\int_{-\pi}^tf(s)\,ds\,dt+cx+d,
\]
where $c$ and $d$ are chosen to make $-u'(-\pi)+\alpha u(-\pi)=u'(\pi)+\alpha u(\pi)=0$.
\end{proof}

We also have an existence and uniqueness result for the Neumann problem. The proof is similar to that of the Robin result.
\begin{proposition}[Neumann Existence and Uniqueness]\label{Prop:Nuniq}
Let $f\in L^1[-\pi,\pi]$ with $\int_{-\pi}^{\pi}f\,dx=0$. A unique $u\in C^1[-\pi,\pi]$ exists satisfying
\begin{itemize}
\item[1.] $u'$ is absolutely continuous on $[-\pi,\pi]$.
\item[2.] $-u''=f$ a.e. on $(-\pi,\pi)$.
\item[3.] $u'(-\pi)=u'(\pi)=0$.
\item[4.] $\int_{-\pi}^{\pi}u\,dx=0$.
\end{itemize}
\end{proposition}

Thus, the solutions $u$ and $v$ in Theorems \ref{Th:RComp} and \ref{Th:NComp} are guaranteed to exist and be unique. For the Robin problem, we in fact prove a bit more. Namely, we show that solutions are obtained by integration against an explicitly computable Green's function.

\begin{proposition}[Green's Representation]\label{Prop:RGreen}
For $\alpha>0$, the Green's function for the Robin problem on the interval $[-\pi,\pi]$ equals
\[
G(x,y)=-\frac{1}{2}c_{\alpha}xy-\frac{1}{2}|x-y|+\frac{1}{2c_{\alpha}},\qquad x,y\in [-\pi,\pi],
\]
where
\begin{equation}\label{eq:calphadef}
c_{\alpha}=\frac{\alpha}{1+\alpha \pi}.
\end{equation}
That is,
\begin{itemize}
\item[1.] $-G_{xx}(x,y)=\delta_x(y),$ for  $x,y\in (-\pi,\pi),$ 
\item[2.] $-G_x(-\pi,y)+\alpha G(-\pi,y)=G_x(\pi,y)+\alpha G(\pi,y)=0$ for  $y\in (-\pi,\pi)$.
\end{itemize}
Thus, if $f\in L^1[-\pi,\pi]$ and $u$ solves
\[
-u''=f  \quad \textup{in} \quad (-\pi,\pi), \qquad -u'(-\pi)+\alpha u(-\pi)=u'(\pi)+\alpha u(\pi)=0,
\]
then
\[
u(x)=\int_{-\pi}^{\pi}G(x,y)f(y)\,dy, \qquad x\in [-\pi,\pi].
\]
\end{proposition}

\begin{proof}
Properties 1 and 2 follow from a straightforward calculation. Define
\begin{equation}\label{eq:wdefGreen}
w(x)=\int_{-\pi}^{\pi}G(x,y)f(y)\,dy.
\end{equation}
We show that $w\in C^1[-\pi,\pi]$ and that $w$ satisfies all three properties of Proposition \ref{Prop:Runiq}.

The Dominated Convergence Theorem gives
\begin{align}
w'(x)&=\int_{-\pi}^{\pi}G_x(x,y)f(y)\,dy \label{eq:wpdefGreen}\\
&=\int_{-\pi}^x\left(-\frac{1}{2}c_{\alpha}y-\frac{1}{2}\right)f(y)\,dy + \int_{x}^{\pi}\left(-\frac{1}{2}c_{\alpha}y+\frac{1}{2}\right)f(y)\,dy,\nonumber
\end{align}
and this representation shows that $w'$ is absolutely continuous on $[-\pi,\pi]$ with $-w''=f$ a.e. Formulas \eqref{eq:wdefGreen} and \eqref{eq:wpdefGreen} for $w$ and $w'$ together with property 2 of the present proposition give $-w'(-\pi)+\alpha w(-\pi)=w'(\pi)+\alpha w(\pi)=0$. The result now follows from uniqueness in Proposition \ref{Prop:Runiq}.

\end{proof}

To prove our main results, we will also need several tools from symmetrization. We start with the decreasing and symmetric decreasing rearrangements.

\begin{definition}
[Decreasing and Symmetric Decreasing Rearrangements]\label{def:decrearr}Suppose $X\subseteq \mathbb{R}$ is a measurable set and $f\in L^{1}(X)$ satisfies the finiteness condition
\[
|\{x\in X:f(x)>t\}|<\infty,\qquad t>\underset{X}{\textup{ess\ inf}} \ f.
\]

Define $f^{\ast}:[0,|X|]\rightarrow[-\infty,+\infty]$
via
\[
f^{\ast}(t)=\begin{cases}
\underset{X}{\textup{ess\ sup}}\ f & \textup{if}\ t=0,\\
\inf\{s:|\{x:s<f(x)\}|\leq t\} & \textup{if}\ t\in(0,|X|),\\
\underset{X}{\textup{ess\ inf}}\ f & \textup{if}\ t=|X|.
\end{cases}
\]
We call $f^{\ast}$ the decreasing rearrangement of $f$. The symmetric decreasing rearrangement of $f$ is the function $f^{\#}:\left[-\frac{1}{2}|X|,\frac{1}{2}|X|\right]\rightarrow[-\infty,+\infty]$ defined by $f^{\#}(t)=f^{\ast}(2|t|)$.
\end{definition}

We next define the notion of a star function, first introduced by Baernstein to solve extremal problems in complex analysis. (For more on star functions and their use in analysis, see \cite{Baernstein Edrei's Spread Conjecture}, \cite{Baernstein Integral means}, \cite{Baernstein how the star function}, \cite{Barenstein Star Function in Complex Analysis}).

\begin{definition}[Star Function]\label{Def:StarFunction} Let $f\in L^{1}(X)$, where $X\subseteq \mathbb{R}$ is a measurable set of finite length. We define the star function of $f$ on the interval $[0,|X|]$ by the formula
\[
f^{\bigstar}(t) =  \underset{|E|=t}{\sup}\ {\displaystyle \int_{E}f\,dx},
\]
where the $\sup$ is taken over all measurable subsets $E\subseteq X$ with $|E|=t$.
\end{definition}
Our next proposition establishes a key connection between the star function and the decreasing rearrangement. 
\begin{proposition}\label{prop:genstarprop}
\label{prop:Star function achieved}Assume $f\in L^{1}(X)$ with $X\subseteq \mathbb{R}$ a measurable subset of finite length. Then for each $t\in[0,|X|]$,
\[
f^{\bigstar}(t)=\int_0^tf^{\ast}(s)\,ds,
\]
where $f^{\ast}$ is the decreasing rearrangement of $f$.
\end{proposition}

For a proof of Proposition \ref{prop:genstarprop}, see Proposition 9.2 of \cite{Baernstein Symmetrization in Analysis}.

Our proofs of Theorems \ref{Th:RComp} and \ref{Th:NComp} will show that the solutions $u$ and $v$ satisfy the star function inequality $u^{\bigstar}\leq v^{\bigstar}$. Our next result recasts this inequality into an equivalent inequality about convex means. 

\begin{proposition}
[Majorization]\label{Prop:Majorization}
Assume $X\subseteq \mathbb{R}$ is a measurable subset of finite length and $u,v\in L^{1}(X)$. Then
\[
u^{\bigstar}  \leq  v^{\bigstar}
\]
on $[0,|X|]$ if and only if
\[
\int_{X}\phi(u)\,dx  \leq  \int_{X}\phi(v)\,dx
\]
for every increasing convex function $\phi:\mathbb{R}\rightarrow\mathbb{R}$. If $\int_{X}u\,dx=\int_{X}v\, dx$, then the word ``increasing'' may be removed from the previous statement.
\end{proposition}
For a proof of Proposition \ref{Prop:Majorization}, see Propositions 10.1 and 10.3 of \cite{Baernstein Symmetrization in Analysis}.

If the convex means of $u$ are dominated by those of $v$ and additional information is known about $u$ and $v$, we can deduce further inequalities about $L^p$-norms, $\esssup$, $\essinf$, and $\osc$.

\begin{corollary}\label{cor:Lpnormsgen}
\label{cor:consequences of u star leq v star} Say $u,v\in L^{1}(X)$ where $X\subseteq \mathbb{R}$ is a measurable subset of finite length and assume $u^{\bigstar}\leq v^{\bigstar}$ on $[0,|X|]$. If either $u,v\geq 0$ or $\int_{X}u\,dx=\int_{X}v\,dx$, then
\[
\|u\|_{L^{p}(X)}  \leq  \|v\|_{L^{p}(X)},\quad1\leq p\leq +\infty.
\]
If $\int_{X}u\,dx=\int_{X}v\,dx$, moreover
\[
\underset{X}{\esssup}\ u \leq  \underset{X}{\esssup}\ v, \qquad 
\underset{X}{\essinf}\ v \leq  \underset{X}{\essinf}\ u, \qquad
\underset{X}{\osc}\ u \leq  \underset{X}{\osc}\ v,
\]
where $\osc=\esssup - \essinf$.
\end{corollary}

We end the background section with three rearrangement inequalities. These inequalities play a major role in our proofs of Theorems \ref{Th:RComp} and \ref{Th:NComp}. The first two are well known (for discussion and proofs see Theorem 1.2.2 of \cite{Kesevan} and Theorem 8.4 of \cite{Baernstein Symmetrization in Analysis}). The third rearrangement inequality appears to be less well known (for discussion and proof see \cite{BaernsteinS1} or Theorem 8.1 of \cite{Baernstein Symmetrization in Analysis}).
\begin{theorem}[Hardy-Littlewood]\label{th:HL}
Given $f\in L^1[-\pi,\pi]$ and  $g\in L^{\infty}[-\pi,\pi]$, we have
\[
\int_{-\pi}^{\pi}fg\,dx\leq \int_{-\pi}^{\pi}f^{\#}g^{\#}\,dx,
\]
with $\#$ the symmetric decreasing rearrangement.
\end{theorem}

\begin{theorem}[Riesz-Sobolev]\label{th:RS}
Suppose $f,g,h\in L^1(\mathbb{R})$ are nonnegative. Then we have
\[
\int_{-\infty}^{\infty}\int_{-\infty}^{\infty}f(x)g(y)h(x-y)\,dy\,dx\leq \int_{-\infty}^{\infty}\int_{-\infty}^{\infty}f^{\#}(x)g^{\#}(y)h^{\#}(x-y)\,dy\,dx,
\]
with $\#$ the symmetric decreasing rearrangement.
\end{theorem}

\begin{theorem}[Baernstein]\label{th:Baernstein}
Let $f,g\in L^1[-\pi,\pi]$ and $h\in L^{\infty}[-\pi,\pi]$ be $2\pi$-periodic functions on $\mathbb{R}$. Then
\[
\int_{-\pi}^{\pi}\int_{-\pi}^{\pi}f(x)g(y)h(x-y)\,dy\,dx\leq \int_{-\pi}^{\pi}\int_{-\pi}^{\pi}f^{\#}(x)g^{\#}(y)h^{\#}(x-y)\,dy\,dx,
\]
with $\#$ the $2\pi$-periodic extension of symmetric decreasing rearrangement on $[-\pi,\pi]$ to all of $\mathbb{R}$.
\end{theorem}

\section{Proofs of Main Results}

\subsection*{The Robin problem}

We start this section with a proof of our first main result.

\begin{proof}[Proof of Theorem \ref{Th:RComp}]
Say $E\subseteq [-\pi,\pi]$ is a measurable subset. Then
\[
\int_Eu(x)\,dx=\int_{-\pi}^{\pi}\int_{-\pi}^{\pi}\chi_E(x)f(y)G(x,y)\,dy\,dx,
\]
where $G$ is the Robin Green's function from Proposition \ref{Prop:RGreen}. Observe that
\begin{align*}
G(x,y)&=-\frac{1}{2}c_{\alpha}xy-\frac{1}{2}|x-y|+\frac{1}{2c_{\alpha}}\\
&=-\frac{1}{4}c_{\alpha}(x^2+y^2)+\frac{1}{4}c_{\alpha}(x-y)^2-\frac{1}{2}|x-y|+\frac{1}{2c_{\alpha}}
\end{align*}
where $c_\alpha$ is the constant given in (\ref{eq:calphadef}).  Since $f,f^{\#}$ are rearrangements and nonnegative, applying Theorem \ref{th:HL} to the $dx$ integral yields
\begin{equation}\label{eq:term1}
\frac{1}{4}c_{\alpha}\int_{-\pi}^{\pi}\int_{-\pi}^{\pi}\chi_E(x)f(y)(-x^2)\,dy\,dx \leq \frac{1}{4}c_{\alpha}\int_{-\pi}^{\pi}\int_{-\pi}^{\pi}\chi_E^{\#}(x)f^{\#}(y)(-x^2)\,dy\,dx.
\end{equation}
Similarly, $\chi_E,\chi_E{^{\#}}$ are rearrangements, so applying Theorem \ref{th:HL} to the $dy$ integral gives
\begin{equation}\label{eq:term2}
\frac{1}{4}c_{\alpha}\int_{-\pi}^{\pi}\int_{-\pi}^{\pi}\chi_E(x)f(y)(-y^2)\,dy\,dx\leq \frac{1}{4}c_{\alpha}\int_{-\pi}^{\pi}\int_{-\pi}^{\pi}\chi_E^{\#}(x)f^{\#}(y)(-y^2)\,dy\,dx.
\end{equation}

Next, write
\[
h(z)=\frac{1}{4}c_{\alpha}z^2-\frac{1}{2}|z|,\qquad z\in [-2\pi,2\pi].
\]
Write $\tilde h$ for the $2\pi$-periodic extension of $h\big|_{[-\pi,\pi]}$ to all of $\mathbb{R}$. When $z\in(0,\pi]$, note that
\[
h'(z)=\frac{1}{2}c_{\alpha}z-\frac{1}{2}\leq \frac{1}{2}c_{\alpha}\pi-\frac{1}{2}=-\frac{1}{2(1+\alpha \pi)}<0.
\]
Thus, $h$ is symmetric decreasing on $[-\pi,\pi]$. It follows from Theorem \ref{th:Baernstein} that
\begin{equation}\label{eq:UseBaernstein}
\int_{-\pi}^{\pi}\int_{-\pi}^{\pi}\chi_E(x)f(y)\tilde h(x-y)\,dy\,dx \leq \int_{-\pi}^{\pi}\int_{-\pi}^{\pi}\chi_E^{\#}(x)f^{\#}(y)\tilde h(x-y)\,dy\,dx.
\end{equation}
Moreover, on $[\pi,2\pi]$,
\begin{align*}
h(z)-\tilde h(z)&= h(z)-h(2\pi-z)\\
&=\frac{1}{4}c_{\alpha}z^2-\frac{1}{2}z-\left(\frac{1}{4}c_{\alpha}(2\pi-z)^2-\frac{1}{2}(2\pi-z)\right)\\
&=(z-\pi)(c_{\alpha}\pi-1).
\end{align*}
As we saw above, $c_{\alpha}\pi-1<0$. And since $h=\tilde h$ on $[0,\pi]$, it follows that $h-\tilde h$ is symmetric decreasing on $[-2\pi,2\pi]$. Extend $\chi_E$ and $f$ to vanish outside $[-\pi,\pi]$ and extend $h-\tilde h+\pi(1-c_{\alpha}\pi)$ to vanish outside $[-2\pi,2\pi]$. Then Theorem \ref{th:RS} gives
\begin{align*}
&\int_{-\pi}^{\pi}\int_{-\pi}^{\pi}\chi_E(x)f(y)\left((h-\tilde h)(x-y)+\pi(1-c_{\alpha}\pi)\right)\,dy\,dx\\
&\qquad \qquad \qquad \qquad  \leq \int_{-\pi}^{\pi}\int_{-\pi}^{\pi}\chi_E^{\#}(x)f^{\#}(y)\left((h-\tilde h)(x-y)+\pi(1-c_{\alpha}\pi)\right)\,dy\,dx.
\end{align*}
Since $\chi_E,\chi_E{^{\#}}$ and $f,f^{\#}$ are rearrangements, the inequality above gives
\begin{equation}\label{eq:UseRS}
\int_{-\pi}^{\pi}\int_{-\pi}^{\pi}\chi_E(x)f(y)(h-\tilde h)(x-y)\,dy\,dx \leq \int_{-\pi}^{\pi}\int_{-\pi}^{\pi}\chi_E^{\#}(x)f^{\#}(y)(h-\tilde h)(x-y)\,dy\,dx.
\end{equation}
Combining inequalities \eqref{eq:UseBaernstein} and \eqref{eq:UseRS} and noting $\chi_E^{\#}=\chi_{E^{\#}}$ gives
\begin{equation}\label{eq:term3}
\int_{-\pi}^{\pi}\int_{-\pi}^{\pi} \chi_{E}(x)f(y)h(x-y)\,dy\,dx \leq \int_{-\pi}^{\pi}\int_{-\pi}^{\pi} \chi_{E^{\#}}(x)f^{\#}(y)h(x-y)\,dy\,dx.
\end{equation}
Finally, note that
\begin{equation}\label{eq:term4}
\frac{1}{2c_{\alpha}}\int_{-\pi}^{\pi}\int_{-\pi}^{\pi}\chi_{E}(x)f(y)\,dy\,dx=\frac{1}{2c_{\alpha}}\int_{-\pi}^{\pi}\int_{-\pi}^{\pi}\chi_{E^{\#}}(x)f^{\#}(y)\,dy\,dx,
\end{equation}
again as $\chi_E,\chi_E{^{\#}}$ and $f,f^{\#}$ are rearrangements. Combining \eqref{eq:term1}, \eqref{eq:term2}, \eqref{eq:term3}, and \eqref{eq:term4} shows
\[
\int_{E}u(x)\,dx\leq \int_{E^{\#}}v(x)\,dx.
\]
Taking the $\sup$ over all measurable subsets $E\subseteq [-\pi,\pi]$ with fixed length, say $t$, the inequality above gives
\[
u^{\bigstar}(t) \leq \int_{-\frac{t}{2}}^{\frac{t}{2}}v(x)\,dx \leq v^{\bigstar}(t).
\]
The theorem's claims about convex means now follows from Proposition \ref{Prop:Majorization}. To prove the remaining claims, we next argue that $u,v\geq 0$. First note that $u$ is concave, so $\underset{[-\pi,\pi]}{\min}\ u$ is achieved at either $-\pi$ or $\pi$. Suppose the minimum occurs at $\pi$. If $u(\pi)<0$, then from the Robin boundary condition we see
\[
u'(\pi)=-\alpha u(\pi)>0,
\]
and so $\underset{[-\pi,\pi]}{\min}\ u$ cannot be achieved at $\pi$. We conclude that $u(\pi)\geq 0$. An identical argument applies when the minimum occurs at $-\pi$ from which we conclude that $u$ is nonnegative.  Since the same argument applies to the function $v,$ we conclude that $v$ is also nonnegative. The theorem's remaining conclusions now follow from Proposition \ref{cor:Lpnormsgen}.
\end{proof} 

\subsection*{From Robin to Dirichlet} As mentioned in the introduction, our Robin comparison principle implies the corresponding Dirichlet result. The main idea behind the proof is that as $\alpha$ tends to $+\infty$, the Robin boundary condition converts into a Dirichlet condition.

\begin{corollary}[ODE Dirichlet Comparison Principle]\label{Cor:DComp}
Let $0\leq f\in L^1[-\pi,\pi]$ and suppose $u$ and $v$ solve the Poisson problems
\begin{align*}
-u''&=f  \quad \textup{in} \quad (-\pi,\pi), \qquad u(-\pi)=u(\pi)=0,\\
-v''&=f^{\#}  \quad \textup{in} \quad (-\pi,\pi), \qquad v(-\pi)=v(\pi)=0,
\end{align*}
with $f^{\#}$ the symmetric decreasing rearrangement of $f$. Then
\[
\int_{-\pi}^{\pi}\phi(u)\,dx\leq \int_{-\pi}^{\pi}\phi(v)\,dx
\]
for each increasing convex function $\phi:\mathbb{R}\to \mathbb{R}$.
\end{corollary}

\begin{proof}
Let $f,u$, and $v$ be as stated, and suppose $u_k$ and $v_k$ solve the Robin problems
\begin{align*}
-u_k''&=f  \quad \textup{in} \quad (-\pi,\pi), \qquad -u_k'(-\pi)+\alpha_k u_k(-\pi)=u_k'(\pi)+\alpha_k u_k(\pi)=0,\\
-v_k''&=f^{\#}  \quad \textup{in} \quad (-\pi,\pi), \qquad -v_k'(-\pi)+\alpha_k v_k(-\pi)=v_k'(\pi)+\alpha_k v_k(\pi)=0,
\end{align*}
where $0<\alpha_k\to +\infty$. Then by Theorem \ref{Th:RComp} and Proposition \ref{Prop:Majorization},
\begin{equation}\label{eq:ukvkstar}
u_k^{\bigstar}(t)\leq v_k^{\bigstar}(t),\qquad t\in [0,2\pi].
\end{equation}
If $G_{\alpha_k}$ denotes the Robin Green's function of Proposition \ref{Prop:RGreen} with parameter $\alpha_k$, then
\[
G_{\alpha_k}(x,y)\to G(x,y)
\]
uniformly on $[-\pi,\pi]\times[-\pi,\pi]$, where
\[
G(x,y)=-\frac{1}{2\pi}xy-\frac{1}{2}|x-y|+\frac{\pi}{2}
\]
is the Dirichlet Green's function on $[-\pi,\pi]$. Using the Green's representation, we see $u_k\to u$ and $v_k\to v$ uniformly on $[-\pi,\pi]$. Since symmetrization decreases the $L^1$-distance, we note that for any $t\in[0,2\pi]$,
\begin{align*}
|u_k^{\bigstar}(t)-u^{\bigstar}(t)|&= \left|\int_{-\frac{t}{2}}^{\frac{t}{2}}\left(u_k^{\#}(x)-u^{\#}(x)\right)\,dx\right|\\
&\leq \int_{-\pi}^{\pi}\left|u_k^{\#}(x)-u^{\#}(x)\right|\,dx\\
&\leq \int_{-\pi}^{\pi}\left|u_k(x)-u(x)\right|\,dx,
\end{align*}
and this last term tends to zero by uniform convergence. Thus we have pointwise convergence of star functions:
\begin{equation}\label{eq:ukstar}
u^{\bigstar}_k(t)\to u^{\bigstar}(t),\qquad t\in[0,2\pi].
\end{equation}
A similar argument shows
\begin{equation}\label{eq:vkstar}
v^{\bigstar}_k(t)\to v^{\bigstar}(t),\qquad t\in[0,2\pi].
\end{equation}
Combining inequality \eqref{eq:ukvkstar} with \eqref{eq:ukstar} and \eqref{eq:vkstar} gives $u^{\bigstar}\leq v^{\bigstar}$ in $[0,2\pi]$, and this inequality is equivalent to the corollary's conclusion courtesy of Proposition \ref{Prop:Majorization}.

\end{proof}

The conclusion of Corollary \ref{Cor:DComp} can be strengthened to $u^{\#} \leq v$ in $[-\pi,\pi]$, though this stronger conclusion is not needed for our paper. The argument is simple but requires additional tools developed by Baernstein. We include the proof for the sake of completeness.

Define
\[
w(s)=u^{\bigstar}(2s)-v^{\bigstar}(2s),\qquad s\in[0,\pi].
\]
Then by Theorem 9.20 in \cite{Baernstein Symmetrization in Analysis} or Theorem 5 of \cite{Baernstein Cortona Volume}, $\frac{d^2}{ds^2}w(s)\geq 0$ weakly. That is,
\[
\int_{0}^{\pi}w(s)G''(s)\,ds \geq 0
\]
for each $G_c^2(0,\pi)$ nonnegative with compact support. Integrating by parts, we see
\[
\int_{0}^{\pi}w'(s)G'(s)\,ds \leq 0.
\]
An easy argument gives that $w'(s)$ is increasing. But
\[
\frac{d}{ds}w(s)=\frac{d}{ds}\int_{-s}^s\left(u^{\#}(x)-v^{\#}(x)\right)\,dx=2\left(u^{\#}(s)-v^{\#}(s)\right)
\]
and so we see that $u^{\#}(s)-v^{\#}(s)$ is increasing for $s\in[0,\pi]$. But as $u^{\#}(\pi)=v^{\#}(\pi)=0$, we conclude  $u^{\#}(s)-v^{\#}(s)\leq 0$. Finally, since $v=v^{\#}$, this last inequality implies $u^{\#}\leq v$ on $[-\pi,\pi]$.

\subsection*{An interesting example} As noted in the introduction, with Robin problems the oscillation need not increase under symmetric decreasing rearrangement. Consider the following example.

\begin{example}\label{ex:osc}
Consider the solutions $u$ and $v$ to the Poisson problems of Theorem \ref{Th:RComp} with $f=\chi_{[-\pi,0]}$:
\begin{align*}
-u''&=\chi_{[-\pi,0]}  \quad \textup{in} \quad [-\pi,\pi], \qquad -u'(-\pi)+\alpha u(-\pi)=u'(\pi)+\alpha u(\pi)=0,\\
-v''&=\chi_{\left[-\frac{\pi}{2},\frac{\pi}{2}\right]}  \quad \textup{in} \quad [-\pi,\pi], \qquad -v'(-\pi)+\alpha v(-\pi)=v'(\pi)+\alpha v(\pi)=0.
\end{align*}
\end{example}
It is straightforward to check that
\[
u(x)=-u_1(x)+\left(\frac{\pi}{2}+\frac{\pi^2}{4}c_{\alpha}\right)x+\frac{\pi}{2c_{\alpha}}+\frac{\pi^2}{4},
\]
where
\[
u_1(x)=
\begin{cases}
\frac{1}{2}x^2+\pi x+\frac{\pi^2}{2} & \textup{if } -\pi\leq x< 0,\\
\pi x+\frac{\pi^2}{2} & \textup{if } 0\leq x\leq \pi,
\end{cases}
\]
and $c_{\alpha}$ is defined in \eqref{eq:calphadef}. Similarly,
\[
v(x)=-v_1(x)+\frac{\pi}{2}x+\frac{\pi}{2c_{\alpha}},
\]
where
\[
v_1(x)=
\begin{cases}
0 & \textup{if }-\pi \leq x<-\frac{\pi}{2},\\
\frac{1}{2}x^2+\frac{\pi}{2}x+\frac{\pi^2}{8} & \textup{if } -\frac{\pi}{2} \leq x < \frac{\pi}{2},\\
\pi x & \textup{if } \frac{\pi}{2}\leq  x\leq \pi.
\end{cases}
\]
Since $v$ is symmetric decreasing, we have
\[
\underset{[-\pi,\pi]}{\osc}\ v=v(0)-v(\pi)=\frac{3\pi^2}{8}.
\]
Note, rather curiously, that this oscillation is independent of $\alpha$. On the other hand,
\[
\underset{[-\pi,\pi]}{\osc}\ u \geq u\left(-\frac{\pi}{2}\right)-u(\pi)=\frac{\pi^2(5+2\alpha \pi)}{8(1+\alpha \pi)}.
\]
Now it is easy to verify that $\frac{3\pi^2}{8}<\frac{\pi^2(5+2\alpha \pi)}{8(1+\alpha \pi)}$ so long as $\alpha<\frac{2}{\pi}$, which ensures $\underset{[-\pi,\pi]}{\osc}\ v < \underset{[-\pi,\pi]}{\osc}\ u$.

Example \ref{ex:osc} leads to an interesting open question.

\begin{problem}[Open]\label{Prob:1dRodRobinOpen}
Suppose half of a rod of length $2\pi$ with Robin boundary conditions is heated uniformly, while in the remaining half, heat is neither generated nor absorbed. Where should we place the heat sources to maximize the temperature gap?
\end{problem}

Our intuition suggests that an optimal source in Problem \ref{Prob:1dRodRobinOpen} is an interval of length $\pi$. Assuming this is the case, we have the following proposition.

\begin{proposition}\label{prop:robingapint}
Assume the temperature gap in Problem \ref{Prob:1dRodRobinOpen} is maximized by a source interval of length $\pi$. Then when $0<\alpha\leq \frac{2}{\sqrt{3}\pi}$, the temperature gap is maximized by locating the source interval at either end of the rod. As $\alpha$ increases from $\frac{2}{\sqrt{3}\pi}$ to $+\infty$, the temperature gap is maximized by a source interval that continuously transitions from one end of the rod towards its center.
\end{proposition}

\begin{proof}
Denote $I_b=\left[b-\frac{\pi}{2},b+\frac{\pi}{2}\right]$ for $-\frac{\pi}{2}\leq b\leq \frac{\pi}{2}$. We consider the gap function
\[
\textup{Gap}(\alpha,b)=\max_{[-\pi,\pi]}u-\min_{[-\pi,\pi]}u,\qquad 0<\alpha<\infty, \quad -\frac{\pi}{2}\leq b\leq \frac{\pi}{2},
\]
where $u$ is the solution of the Poisson problem
\[
-u''=\chi_{I_b}  \quad \textup{in} \quad (-\pi,\pi), \qquad -u'(-\pi)+\alpha u(-\pi)=u'(\pi)+\alpha u(\pi)=0.
\]
By symmetry, one has
\[
\textup{Gap}(\alpha,-b)=\textup{Gap}(\alpha,b),\quad 0\leq b\leq \frac{\pi}{2},
\]
and so for the remainder of the argument, we focus our attention on $\textup{Gap}$ when $-\frac{\pi}{2}\leq b\leq 0$, i.e. on source intervals whose center lies in the left half of the rod.

With a bit of work and help from Mathematica, one computes
\[
\textup{Gap}(\alpha,b)=-\frac{\pi(1+\alpha(b+\pi))(-3\pi(1+\alpha\pi)+b(4+3\alpha\pi))}{8(1+\alpha \pi)^2}
\]
and so
\begin{align}
\frac{\partial\textup{Gap}}{\partial b}(\alpha,b)&=-\frac{\pi(2+2\alpha \pi+\alpha b(4+3\alpha \pi))}{4(1+\alpha \pi)^2},\label{eq:gapfirstd}\\
\frac{\partial^2\textup{Gap}}{\partial b^2}(\alpha,b)&=-\frac{\alpha \pi (4+3\alpha\pi)}{4(1+\alpha \pi)^2}.\label{eq:gapconv}
\end{align}
Equation \eqref{eq:gapfirstd} implies that $\textup{Gap}(\alpha,b)$ is decreasing in $b$ on $\left[-\frac{\pi}{2},0\right]$ when $\alpha \in \left(0,\frac{2}{\sqrt{3}\pi}\right]$. This establishes the proposition's first claim.

Equation \eqref{eq:gapconv} says that $\textup{Gap}(\alpha,b)$ is strictly concave in $b$ holding $\alpha$ fixed. We also note that
\[
\frac{\partial \textup{Gap}}{\partial b}\left(\alpha,-\frac{\pi}{2}\right)=\frac{\pi(-4+3\alpha^2\pi^2)}{8(1+\alpha\pi)^2}>0
\]
when $\alpha>\frac{2}{\sqrt{3}\pi}$. It also holds that 
\[
\frac{\partial \textup{Gap}}{\partial b}(\alpha,0)=-\frac{\pi}{2(1+\alpha \pi)}<0,
\]
and so $\textup{Gap}(\alpha,b)$ is maximized at some $b\in \left(-\frac{\pi}{2},0\right)$ when $\alpha>\frac{2}{\sqrt{3}\pi}$. Returning to equation \eqref{eq:gapfirstd}, we note that $\textup{Gap}(\alpha,b)$ is maximized when
\[
b=b_{\textup{crit}}=-\frac{2(1+\alpha \pi)}{\alpha(4+3\alpha \pi)}.
\]
Since
\[\frac{d b_{\textup{crit}}}{d \alpha}= \frac{8+6\alpha \pi(2+\alpha \pi)}{\alpha^2(4+3\alpha \pi)^2}>0,
\]
we conclude that $b_{\textup{crit}}$ continuously increases from $-\frac{\pi}{2}$ towards $0$ as $\alpha$ increases from $\frac{2}{\sqrt{3}\pi}$ to $+\infty$.
\end{proof}

\subsection*{The Neumann problem}
Our approach here is driven by Fourier series. Suppose that $f\in L^1[-\pi,\pi]$ has zero mean. One attempt to solve the Poisson equation $-u''=f$ in $(-\pi,\pi)$ might be to consider the function whose Fourier series is given by
\begin{equation}\label{eq:fourier}
u(x)=\sum_{n\neq 0}\frac{1}{n^2}\hat f(n)e^{inx},
\end{equation}
where $\hat f(n)=\frac{1}{2\pi}\int_{-\pi}^{\pi}f(x)e^{-inx}\,dx$. Formally differentiating the equation above termwise shows
\[
-u''(x)=\sum_{n\neq 0}\hat f(n)e^{inx}=f(x).
\]
We are thus led to consider the function $K$ whose Fourier coefficients are given by
\[
\hat K(n)=
\begin{cases}
\frac{1}{n^2} & \textup{if }n\neq 0,\\
0 & \textup{if }n=0.
\end{cases}
\]
One readily verifies that
\[
K(x)=\frac{1}{2}x^2-\pi |x|+\frac{1}{3}\pi^2, \qquad -\pi \leq x \leq \pi.
\]
Extend $f$ and $K$ to all of $\mathbb{R}$ by $2\pi$-periodicity. The function $u$ in \eqref{eq:fourier} has the property that\
\[
\hat{u}(n)=\hat{K}(n)\hat{f}(n)=\widehat{K\ast f}(n)
\]
which leads us to study the convolution
\begin{equation}\label{eq:udefcont}
u(x)=(K\ast f)(x)=\frac{1}{2\pi}\int_{-\pi}^{\pi}K(x-y)f(y)\,dy.
\end{equation}
We now investigate how the $u$ defined in \eqref{eq:udefcont} differs from the $u$ in Proposition \ref{Prop:Nuniq}. We will see that $u$ defined in \eqref{eq:udefcont} satisfies properties 1, 2, and 4 of Proposition \ref{Prop:Nuniq}. Below, we continue to identify $K$ and $f$ with their $2\pi$-periodic extensions.

\begin{proposition}\label{prop:uconvdef}
The function $u=K\ast f$ defined in \eqref{eq:udefcont} satisfies the following properties:
\begin{enumerate}
\item[1.] $u$ is continuously differentiable on $\mathbb{R}$ with $u'=K'\ast f$.
\item[2.] $u'$ is absolutely continuous on $[-\pi,\pi]$ and
\[
u'(x)-u'(-\pi)=-\int_{-\pi}^xf(y)\,dy,\qquad -\pi \leq x \leq \pi,
\]
with $-u''=f$ a.e. in $(-\pi,\pi)$.
\item[3.] $u'(-\pi)=u'(\pi)=-\frac{1}{2\pi}\int_{-\pi}^{\pi}xf(x)\,dx$.
\item[4.] $\int_{-\pi}^{\pi}u(x)\,dx=0$.
\end{enumerate}
\end{proposition}

\begin{proof}
Property 1 is a standard fact about convolutions. For property 2, we compute
\begin{align*}
(K'\ast f)(x)&=\frac{1}{2\pi}\int_{-\pi}^{\pi}K'(y)f(x-y)\,dy\\
&=\frac{1}{2\pi}\int_{-\pi}^{0}(y+\pi)f(x-y)\,dy+\frac{1}{2\pi}\int_{0}^{\pi}(y-\pi)f(x-y)\,dy\\
&=\frac{1}{2\pi}\int_{x-\pi}^x(x-y-\pi)f(y)\,dy+\frac{1}{2\pi}\int_{x}^{x+\pi}(x-y+\pi)f(y)\,dy,
\end{align*}
and from this representation it follows that $(K'\ast f)(x)$ is absolutely continuous on $[-\pi,\pi]$. Its derivative equals $-f(x)$ by direct calculation, for almost every $x$. Thus property 2 holds. For property 3, one sees from the above calculation
\[
(K'\ast f)(-\pi)=(K'\ast f)(\pi)=-\frac{1}{2\pi}\int_{-\pi}^{\pi}yf(y)\,dy.
\]
To establish property 4 we simply observe
\[
\frac{1}{2\pi}\int_{-\pi}^{\pi}u(x)\,dx=\hat{u}(0)=\hat{K}(0) \hat{f}(0)=0.
\]
\end{proof}

A consequence of Proposition \ref{prop:uconvdef} is:

\begin{proposition}\label{prop:uisconvwf}
Let $f$ and $u$ be as in Proposition \ref{Prop:Nuniq}. If $\int_{-\pi}^{\pi}xf(x)\,dx=0$, then $u$ is given by convolution as in equation \eqref{eq:udefcont}.
\end{proposition}

Before proving Theorem \ref{Th:NComp}, we first prove a preliminary comparison result.

\begin{theorem}\label{Th:Nprelim}
Let $f\in L^1[-\pi,\pi]$ where $\int_{-\pi}^{\pi}f(x)\,dx=\int_{-\pi}^{\pi}xf(x)\,dx=0$. Let $u$ and $v$ solve the Poisson problems
\begin{align*}
-u''&=f  \quad \textup{in} \quad (-\pi,\pi), \qquad u'(-\pi)=u'(\pi)=0,\\
-v''&=f^{\#}  \quad \textup{in} \quad (-\pi,\pi), \qquad v'(-\pi)=v'(\pi)=0,
\end{align*}
where $f^{\#}$ is the decreasing rearrangement of $f$, and $\int_{-\pi}^{\pi}u(x)\,dx=\int_{-\pi}^{\pi}v(x)\,dx=0$. Then
\[
u^{\bigstar}(t)\leq  \int_{-\frac{t}{2}}^{\frac{t}{2}}v(x)\,dx \leq v^{\bigstar}(t)
\]
on $[0,2\pi]$. 
\end{theorem}

\begin{proof}
By Proposition \ref{prop:uisconvwf}, we have $u=K\ast f$ and $v=K\ast f^{\#}$ (because $f^{\#}$ is even, and so $\int_{-\pi}^{\pi}xf^{\#}(x)\,dx=0$). Fix $t\in [0,2\pi]$ and let $E\subseteq [-\pi,\pi]$ denote a measurable subset of length $t$. Applying Theorem \ref{th:Baernstein} we see
\[
\int_Eu\,dx=\int_E K\ast f\,dx \leq \int_{E^{\#}}K\ast f^{\#}\,dx=\int_{E^{\#}}v\,dx.
\]
Taking the $\sup$ in the above inequality over all measurable subsets $E$ of $[-\pi,\pi]$ of length $t$, we obtain the desired conclusion.
\end{proof}

We are now prepared to prove our paper's second main result.

\begin{proof}[Proof of Theorem \ref{Th:NComp}]
Given $f\in L^1[0,\pi]$, extend $f$ to $[-\pi,\pi]$ by even reflection, and denote this extension by $\tilde f$. Observe that $\int_{-\pi}^{\pi}x\tilde f(x)\,dx=0$. Clearly,
\begin{equation}\label{eq:sdrftild}
(\tilde f)^{\#}(x)=f^{\ast}(x),\qquad 0\leq x\leq \pi.
\end{equation}
Let $u$ correspond to $f$ in the analogue of Proposition \ref{Prop:Nuniq} over $[0,\pi]$ and similarly let $v$ correspond to $f^{\ast}$. Let $\tilde u$ and $\tilde v$ correspond to $\tilde f$ and $(\tilde f)^{\#}$ in the $[-\pi,\pi]$ version of Proposition \ref{Prop:Nuniq}. Proposition \ref{prop:uisconvwf} gives $\tilde u = K\ast \tilde f$ and $\tilde v = K \ast (\tilde f)^{\#}$.

We claim that $\tilde u$ is obtained from $u$ by even reflection. We show that $\tilde u$ is even and that $\tilde u$ also satisfies the properties of Proposition \ref{Prop:Nuniq} corresponding to $f$ over the interval $[0,\pi]$. First, $\tilde u$ is even:
\[
\tilde u(x)=\frac{1}{2\pi}\int_{-\pi}^{\pi}K(x-y)\tilde f(y)\,dy=\frac{1}{2\pi}\int_{-\pi}^{\pi}K(x+y)\tilde f(y)\,dy=\frac{1}{2\pi}\int_{-\pi}^{\pi}K(-x-y)\tilde f(y)\,dy=\tilde u(-x).
\]
Since $\tilde u\in C^1[-\pi,\pi]$ and is even, we must have $\tilde u'(0)=0$. Additionally, we have $\tilde u'(\pi)=0$ by assumption. Again, being even implies $\int_{0}^{\pi}\tilde u(x)\,dx=\frac{1}{2}\int_{-\pi}^{\pi}\tilde u(x)\,dx=0$. Hence by uniqueness, $\tilde u(x)=u(x)$ on $[0,\pi]$ and so $(\tilde u)^{\#}(x)=u^{\ast}(x)$ on $[0,\pi]$. We similarly have $\tilde v(x)=v(x)$ on $[0,\pi]$. By Theorem \ref{Th:Nprelim}, we have for each $0\leq t\leq 2\pi$
\[
\int_{-\frac{t}{2}}^{\frac{t}{2}}(\tilde u)^{\#}(x)\,dx \leq \int_{-\frac{t}{2}}^{\frac{t}{2}}\tilde v(x)\,dx
\]
which implies
\[
\int_0^{\frac{t}{2}}(\tilde u)^{\#}(x)\,dx\leq \int_0^{\frac{t}{2}}\tilde v(x)\,dx
\]
finally giving
\[
\int_0^{\frac{t}{2}}u^{\ast}(x)\,dx\leq \int_0^{\frac{t}{2}}v(x)\,dx.
\]
This last inequality implies $u^{\bigstar} \leq v^{\bigstar}$ on $[0,\pi]$ and the theorem's claims on convex means follows from Proposition \ref{Prop:Majorization}. The remaining conclusions follow from Proposition \ref{cor:Lpnormsgen} since $u$ and $v$ have zero mean by assumption.
\end{proof}

\end{document}